\documentclass[12pt]{amsart}

\usepackage{psfrag}
\usepackage{color}

\usepackage{graphicx,graphics}
\usepackage{fullpage,amssymb,amsfonts,amsmath,amstext,amsthm,amscd,enumerate,verbatim,tikz}
\usepackage[T1]{fontenc}
\usetikzlibrary{matrix,arrows}

\usepackage[all]{xy} \usepackage{url}

\begin{document}
\newcommand{\note}[1]{\marginpar{\tiny #1}}
\newtheorem{theorem}{Theorem}[section]
\newtheorem{result}[theorem]{Result}
\newtheorem{fact}[theorem]{Fact}
\newtheorem{example}[theorem]{Example}
\newtheorem{conjecture}[theorem]{Conjecture}
\newtheorem{definition}[theorem]{Definition}
\newtheorem{lemma}[theorem]{Lemma}
\newtheorem{proposition}[theorem]{Proposition}
\newtheorem{remark}[theorem]{Remark}
\newtheorem{corollary}[theorem]{Corollary}
\newtheorem{facts}[theorem]{Facts}
\newtheorem{question}[theorem]{Question}
\newtheorem{props}[theorem]{Properties}

\newtheorem{ex}[theorem]{Example}

\newcommand{\notes} {\noindent \textbf{Notes.  }}
\renewcommand{\note} {\noindent \textbf{Note.  }}
\newcommand{\defn} {\noindent \textbf{Definition.  }}
\newcommand{\defns} {\noindent \textbf{Definitions.  }}
\newcommand{\x}{{\bf x}}
\newcommand{\z}{{\bf z}}
\newcommand{\B}{{\bf b}}
\newcommand{\V}{{\bf v}}
\newcommand{\T}{\mathcal{T}}
\newcommand{\Z}{\mathbb{Z}}
\newcommand{\Hp}{\mathbb{H}}
\newcommand{\D}{\mathbb{D}}
\newcommand{\R}{\mathbb{R}}
\newcommand{\N}{\mathbb{N}}
\renewcommand{\B}{\mathbb{B}}
\newcommand{\C}{\mathbb{C}}
\newcommand{\dt}{{\mathrm{det }\;}}
 \newcommand{\adj}{{\mathrm{adj}\;}}
 \newcommand{\0}{{\bf O}}
 \newcommand{\w}{\omega}
 \newcommand{\av}{\arrowvert}
 \newcommand{\zbar}{\overline{z}}
 \newcommand{\htt}{\widetilde{h}}
\newcommand{\ty}{\mathcal{T}}
\renewcommand\Re{\operatorname{Re}}
\renewcommand\Im{\operatorname{Im}}
\newcommand{\diam}{\operatorname{diam}}
\newcommand{\dist}{\text{dist}}
\newcommand{\ds}{\displaystyle}
\numberwithin{equation}{section}
\newcommand{\cN}{\mathcal{N}}
\renewcommand{\theenumi}{(\roman{enumi})}
\renewcommand{\labelenumi}{\theenumi}
\newcommand{\inte}{\operatorname{int}}
\newcommand{\fix}{\operatorname{Fix}}

\date{\today}
\title{Attractor sets and Julia sets in low dimensions}
\author{A. Fletcher}
\email{fletcher@math.niu.edu}
\address{Department of Mathematical Sciences, Northern Illinois University, Dekalb, IL 60115, USA}

\thanks{This work was supported by a grant from the Simons Foundation (\#352034, Alastair Fletcher).}

\begin{abstract}
If $X$ is the attractor set of a conformal IFS in dimension two or three, we prove that there exists a quasiregular semigroup $G$ with Julia set equal to $X$. We also show that in dimension two, with a further assumption similar to the open set condition, the same result can be achieved with a semigroup generated by one element. Consequently, in this case the attractor set is quasiconformally equivalent to the Julia set of a rational map.
\end{abstract}

\maketitle

\section{Introduction}

\subsection{Background}

Fractals have received much attention over the last few decades as objects where there is interesting structure on all small scales. This stands in contrast to, for example, smooth curves which look more and more like line segments on smaller and smaller scales.
There are various ways to generate fractals. Two of the most common are as attractor sets of iterated function systems (IFSs) and Julia sets of holomorphic or meromorphic functions. Objects such as the Koch snowflake and Sierpinski gasket arise as attractor sets, see \cite{Falconer}, whereas fractals such as the Douady rabbit arise as Julia sets.
In a sense, these two notions are inverse to one another. The attractor set, as the name suggests, arises from the fact that the maps generating the IFS are contractions, whereas the repelling periodic points are dense in a Julia set.

The aim of this note is to show that every attractor set of a conformal IFS is also a Julia set in dimensions two and three. The caveat here is that we will consider Julia sets of quasiregular semigroups. These are a natural generalization of the more well-known theory of rational semigroups. Rational semigroups were first studied by Hinkkanen and Martin \cite{HM}. Quasiregular semigroups have the advantage of being defined in any (real) dimension. 
The Julia set of a quasiregular semigroup is the set of non-normality of the family generated by elements of the semigroup. The point here is that there is a uniform bound on the distortion of the maps in the semigroup, and so the quasiregular version of Montel's Theorem may be applied. Quasiregular semigroups were first studied by Iwaniec and Martin \cite{IM}, with more recent developments paralleling the development of the theory of rational semigroups in \cite{F}.

By realizing the attractor set of an IFS as a Julia set, we can bring the viewpoint of quasiregular semigroups to bear on these attractor sets. We view quasiregular semigroup theory as a generalization of complex dynamics to a family of mapping with distortion, and so we can expect some of the classical complex dynamical results to apply, for example, we immediately recover the fact that the attractor set of an IFS must be closed. For a more refined geometric result, it is well-known that the Julia set of a rational map is uniformly perfect and moreover, from \cite{F} we know that the Julia set of a quasiregular semigroup is also uniformly perfect. Consequently, we obtain the result that the attractor set of an IFS, of the type considered in this paper, is uniformly perfect. This result is contained in the paper of Stankewitz \cite{Stankewitz} in dimension two and is a minor strengthening of results from Xie et al \cite{XYS} in dimension three. This is because the results in \cite{XYS} are stated for IFSs generated by affine maps, and the conformal maps here, while they must be M\"obius, need not be affine. 

It is known that the attractor set of an IFS can have non-empty interior, see for example \cite{HS}. Since the Julia set of one uniformly quasiregular mapping cannot have interior without being all of $\overline{\R^n}$, we do need to consider semigroups generated by more than one element in order to capture all possible attractor sets.

There are several results in the literature along the theme of approximating a given set arbitrarily closely by Julia sets. First, Lindsey \cite{Lindsey} proved that any Jordan curve in $\C$ can be approximated arbitrarily well by the Julia set of a polynomial. This was refined by Lindsey and Younsi in \cite{LY}. Bishop and Pilgrim \cite{BP} showed that any continuum in the plane can be approximated arbitrarily well by dendrite Julia sets. Here, we are specializing to consider attractor sets of IFSs but we are able to obtain equality instead of approximation and, moreover, our arguments work in dimension three too. 

\subsection{Statement of results}
Our main result reads as follows. We denote the open unit ball in $\R^n$ by $\B^n$.

\begin{theorem}
\label{thm:1}
Let $j\in \{2,3\}$.
Let $X$ be a conformal IFS generated by contractions $\varphi_1,\ldots, \varphi_m$ acting on $\overline{\B^j}$ and so that at least two of the maps generating $X$ have distinct fixed points.
Suppose further that $\varphi_i(\B^j)$ is a quasiball that is relatively compact in $\B^j$ for $i=1,\ldots,m$. If $S$ is the attractor set of $X$,
then there exists a quasiregular semigroup $G$ acting on $\R^j$, with $J(G) = S$. Moreover, each element of $G$ can be chosen to have degree two. 
\end{theorem}

The reason for the upper bound on the dimensions considered here is purely due to the extension result of Berstein and Edmonds \cite{BE} being available only in dimension three. Any higher dimensional generalization of their result would have immediate applications here. It is worth noting that in dimension three, our construction of $G$ yields degree two uniformly quasiregular mappings. Uqr mappings of any degree can be constructed via the conformal trap method, see \cite{IM}, but their Julia sets are necessarily tame Cantor sets. Our maps agree with a degree two power map in a neighbourhood of infinity which, in particular, have a superattracting fixed point at infinity.

A quasiball is the image of the unit ball under an ambient quasiconformal mapping. In dimension three, this hypothesis is vacuous since the hypothesis that we consider conformal maps implies these maps are M\"obius. In dimension two, this requirement makes sure that we are able to continue functions defined on $\varphi_i(\B^2)$ to the boundary.

If we write $\fix(f)$ for the set of fixed points of a map $f$, then the condition that $\varphi_i(\B^j)$ is relatively compact in $\B^j$ means that $\fix(\varphi_i)$ is non-empty in $\B^j$.

If $X$ is a conformal IFS acting on $\C$, then every element of $X$ is necessarily a linear map of the form $\varphi(x) =ax+b$ with $|a|<1$. It follows that there exists $R>0$ so that $X$ can be viewed as an IFS acting on $\overline{B(0,R)}$. By conjugating by a dilation, we can assume $R=1$ and apply the above result. The generalized Liouville Theorem (see \cite[p.85]{IMbook}) implies we can make the same observation in dimension three.

Every uniformly quasiregular map in dimension two is quasiconformally conjugate to a holomorphic function, see \cite{S2,Tukia}. However, it is not true that every quasiregular semigroup in dimension two is quasiconformally conjugate to a semigroup of holomorphic functions, see \cite{H}, so we cannot conclude from Theorem \ref{thm:1} that $X$ is quasiconformally conjugate to the Julia set of a rational semigroup.

One can ask if Theorem \ref{thm:1} can be obtained for the Julia set of a single uqr map. We will restrict ourselves to proving this is true in a certain case. 

\begin{definition}
\label{def:1}
We say that a conformal IFS $X$ generated by $\varphi_1,\ldots, \varphi_m$ acting on $\overline{\B^2}$ satisfies the strong disk open set condition if 
\begin{enumerate}[(i)]
\item $\overline{\varphi_i(\B^2)} \subset \B^2$ for $i=1,\ldots, m$,
\item $\overline{\varphi_i(\B^2)} \cap \overline{\varphi_j(\B^2)}  = \emptyset$ for $i\neq j$.
\end{enumerate}
\end{definition}

This is a slight strengthening of the well-known open set condition and ensures there is a definite gap between the various images $\varphi_i(\B^2)$. This allows interpolation of mappings between them. The trade-off here is that $S$ must necessarily be a Cantor set so the following result does not cover, for example, the Koch snowflake.

\begin{theorem}
\label{thm:2}
If a finitely generated conformal IFS $X$ acting on $\overline{\B^2}$ satisfies the strong disk open set condition, then there exists a uqr map $f:\C \to \C$ with $J(f)$ equal to the attractor set $S$ of $X$.
\end{theorem}

As an immedate consequence of the fact that every uqr map on $\C$ is quasiconformally conjugate to a holomorphic function, we have the following.

\begin{corollary}
\label{cor:1}
If a finitely generated conformal IFS $X$ acting on $\overline{\B^2}$ satisfies the strong disk open set condition, then there exists a polynomial $p$ and a quasiconformal map $h:\C \to \C$ so that the attractor set $S$ of $X$ is equal to $h(J(p))$.
\end{corollary}

To extend Theorem \ref{thm:2} to $\B^3$, we would need to replace $z\mapsto z^m$ on a neighbourhood of infinity with uqr power maps in $\R^3$. These do not necessarily exist for every degree $m\geq 2$. Consequently, one could conceivably obtain an analogue to Theorem \ref{thm:2} in dimension three for an IFS which is generated by $m$ conformal maps, where there exists a power map of degree $m$. 

Instead of considering an IFS acting on $\overline{\B^3} \subset \R^3$, if we consider an IFS acting on a solid torus in $\R^3$,
we can obtain IFSs where the attractor set is a wild Cantor set such as Antoine's necklace. There do exist uqr maps in $\R^3$ where the Julia set is an Antoine's necklace, but as far as the author is aware there is no systematic method for showing that wild Cantor sets are or are not Julia sets of uqr maps.

The paper is organized as follows. In section 2, we recall preliminary material on IFSs and quasiregular semigroups. In section 3, we prove the results in dimension two. In section 4, we prove the results in dimension three.

The author would like to thank Dave Sixsmith for pointing out a typo, and for the referee for numerous comments that helped to improve the exposition.

\section{Preliminaries}

\subsection{IFSs}

A finitely generated iterated function system (IFS) $X$ acting on a closed subset $D$ of $\R^n$ is a collection of contractions $\{ \varphi_1 , \ldots, \varphi_m \}$ with $\varphi_i:D \to D$ for $i=1,\ldots, m$. This means that, for $i=1,\ldots, m$, there is a constant $s_i \in (0,1)$ so that $|\varphi_i(x) - \varphi_i(y) | \leq s_i|x-y|$ for all $x,y\in D$. We call the smallest such constant $s_i$ the contraction factor of $\varphi_i$.
By a famous result of Hutchinson \cite{Hutchinson}, there then exists a unique non-empty compact set $S \subset D$ which is invariant under $X$, that is,
\[ S = \bigcup_{i=1}^m \varphi_i(S).\]
Moreover, if we define the $k$'th iterate of $X$ on a compact set $E$ via $X^0(E) = E$, $X^1(E) = \bigcup_{i=1}^m\varphi_i(E)$ and $X^k(E) = X(X^{k-1}(E))$ for $k\in \N$, then
\[ S = \bigcap _{k=1}^{\infty} X^k(D).\]
The set $S$ is called the attractor set for $X$. Attractor sets are a good way to generate fractals such as the von Koch snowflake, or Sierpinski gasket. If the $\varphi_i$ are conformal mappings, we say $X$ is a conformal IFS.
We refer to \cite{Falconer} for more details on IFSs, attractor sets and fractals.

\subsection{Julia sets and rational semigroups}

Another way to generate fractal behaviour is through the Julia sets of rational maps. If $f:\overline{\C} \to \overline{\C}$ is a rational map, then the Julia set is defined as follows: $z \in J(f)$ if and only if there is no neighbourhood $U$ of $z$ so that the family $\{ f^m|U : m\in \N\}$ is normal. Here, $f^m$ denotes the $m$'th iterate of $f$. It is well-known that $J(f)$ is a non-empty, closed, perfect subset of $\overline{\C}$ and there are examples, such as Latt\`es maps, where $J(f) = \overline{\C}$.
In fact, if $J(f)$ contains an open set, it must be all of $\overline{\C}$. The Fatou set is the complement of the Julia set. See for example Milnor's book \cite{Milnor} for an introduction to rational dynamics.

A rational semigroup $G$ is generated by a collection (assumed in this paper to be finite) of rational maps $\{ f_1,\ldots, f_m\}$ acting on $\overline{\C}$. The Julia set $J(G)$ and the Fatou set $F(G)$ can be defined via normality in exactly the same way as for a single rational map. Namely, $z\in F(G)$ if and only if there exists a neighbourhood $U$ of $z$ so that $G|_U$ is a normal family and $z\in J(G)$ if and only if there is no neighbourhood $U$ of $z$ for which $G|_U$ is a normal family.
There are differences in the properties for Julia sets of rational maps and for Julia sets of rational semigroups. For example, \cite[Example 1]{HM} shows how the Julia set of a rational semigroup can have non-empty interior, and yet not be all of $\overline{\C}$.

\subsection{Quasiregular semigroups}

The natural way to extend the notion of rational semigroups into higher real dimensions is through uniform quasiregularity. Briefly, if $n\geq 2$, a continuous map $f:\R^n \to \R^n$ is called \emph{quasiregular} if $f\in ACL^n$, and there exists $K\geq 1$ so that $|f'(x)| ^n \leq KJ_f(x)$ almost everywhere. Quasiregular maps are sometimes called mappings of bounded distortion, and this captures the main property of such mappings. The smallest $K$ for which the above inequality holds is called the outer distortion $K_O(f)$. If $f$ is quasiregular, then we also have $J_f(x) \leq K' \inf_{|h|=1} |f'(x)h|^n$ almost everywhere. The smallest $K'$ for which this holds is called the inner distortion $K_I(f)$. The maximal distortion of $f$ is $K(f) = \max\{ K_O(f) , K_I(f) \}$. 
We can extend quasiregular mappings to be defined at infinity, or have poles, by composing with suitable M\"obius mappings which send the point at infinity to $0$.
We refer to, for example, \cite{Rickman} for foundational material on quasiregular mappings.

The composition of two quasiregular mappings is again quasiregular, but typically the distortion increases (consider iterating the map $x+iy \mapsto Kx+iy$ for $K>1$).
A uniformly quasiregular mapping, or uqr map for short, is defined by the property that there is a uniform bound on the distortion of the iterates. Trivially, every holomorphic map is uqr.

A quasiregular semigroup $G$ generated by $\{f_1,\ldots, f_m \}$ acting on $\overline{\R^n}$ has the property that there exists $K\geq 1$ so that every element of $G$ is a quasiregular map with maximal distortion at most $K$. We necessarily have that every element of a quasiregular semigroup is a uniformly quasiregular map.. Clearly in dimension two with $K=1$, we just recover rational semigroups.

The Julia set and Fatou set are defined in exactly the same way as for rational semigroups. Here, the quasiregular version of Montel's Theorem plays an important role in studying the properties of these dynamical objects.
See \cite{IM} for the introduction of quasiregular semigroups to the literature and \cite{F} for more recent developments.

\subsection{Quasiconformal Annulus Theorem}

We will repeatedly use Sullivan's Annulus Theorem in the quasiconformal category, see \cite{Sullivan}, and \cite[Theorem 5.8]{TV} for a quantative version. This states that if $A$ is an open annulus in $\R^n$, for $n\geq 2$, then every quasiconformal embedding $f:\R^n \setminus A \to \R^n$ can be extended to a quasiconformal map $\widetilde{f}:\R^n \to \R^n$ with $f|_{\R^n \setminus A'} = f$, where $\overline{A} \subset \operatorname{int}(A')$.

We will apply the annulus theorem when the ring domain has quasispheres as its boundary components. Here, quasispheres are the image of the unit sphere $S^{n-1}$ under an ambient quasiconformal mapping of $\R^n$. While this minor extension of the Annulus Theorem is presumably well-known, we could not find a specific reference and so include a proof for the convenience of the reader. We will need the following weak version of a Uniformization Theorem for rings.

\begin{lemma}
\label{lem:ringdomain}
Let $n\geq 2$ and $0<a<b$. Then if $U,V$ are quasiballs in $\R^n$, with $\overline{V} \subset U$, then there exists a quasiconformal map $f:\R^n\to\R^n$ with $f(U\setminus \overline{V}) =  \{x: a<|x|<b \}$.
\end{lemma}

\begin{proof}
Since $U$ is a quasiball, there exists a quasiconformal map $\varphi:\R^n \to \R^n$ with $\varphi(U) = B(0,b)$. Then there exists $r<b$ so that if $V_1 = \varphi(V)$, then $V_1 \subset \overline{B(0,r)}$. Since $V_1$ is also a quasiball, there exists a quasiconformal map $\psi :\R^n \to \R^n$ with $\psi(V_1) = B(0,a)$. Moreover, given $s\in (\max\{a,r\},b)$, by postcomposing by a radial quasiconformal map of the form
\[ x\mapsto \left \{ \begin{array}{cc} x, & |x| \leq a,\\ (|x|/a)^Kx & |x|>a \end{array} \right . \]
for a suitably small $K>0$, we may assume that $\overline{\psi(B(0,r))} \subset B(0,s)$. 
For $r<r_1<s_1<s$, we apply the Annulus Theorem with $A = \{x: r_1<|x|<s_1\}$ and $A' = \{ x: r<|x|<s\}$ to replace
replace $\psi$ by $\psi_{Ann}$, where $\psi_{Ann}|_{\overline{B(0,r)}} = \psi$ and $\psi_{Ann}|_{\R^n \setminus B(0,s)}$ is the identity. We conclude that $\psi_{Ann} \circ \varphi$ maps $U\setminus \overline{V}$ onto $\{x: a<|x|<b \}$.
\end{proof}

\begin{theorem}
\label{thm:ann}
Let $n\geq 2$, and let $U_1,U_2,V_1,V_2$ be bounded quasiballs with $\overline{V_i} \subset U_i$ for $i=1,2$. If there exists a quasiconformal map $f: \R^n \setminus (U_1 \setminus \overline{V_1}) \to \R^n \setminus ( U_2\setminus \overline{V_2})$, then there exists a quasiconformal map $\widetilde{f} : U_1 \setminus \overline{V_1} \to U_2 \setminus \overline{V_2}$ which agrees with $f$ on both boundary components of $U_1 \setminus \overline{V_1}$.
\end{theorem}

\begin{proof}
By Lemma \ref{lem:ringdomain}, there exists a quasiconformal map $\varphi: \R^n \to \R^n$ with $\varphi( U_1 \setminus \overline{V_1}) = A:= \{x : 1<|x|<2 \}$. The quasiconformal map $f\circ \varphi^{-1}: \R^n \setminus \{x: 1\leq |x| \leq 2 \} \to \R^n$ can be extended to a quasiconformal map $f_{Ann} :\R^n \to \R^n$ by the Annulus Theorem. Moreover, $f_{Ann} |_{\partial A} = (f\circ \varphi^{-1} )|_{\partial A}$. It follows that $\widetilde{f} := f_{Ann}\circ \varphi$ is the required extension of $f$.
\end{proof}

\section{Dimension two}

\subsection{Quadratic polynomials}

In this section we will prove Theorem \ref{thm:1} in dimension two and Theorem \ref{thm:2}. Our construction in dimension two is direct, and the semigroup will be constructed from degree two maps, each of which will glue two of the conformal contractions from $X$ into a quadratic polynomial. We start by considering certain quadratic polynomials.

Let $c\in \B^2$, $a\in \C$ and let $p(z) = a(z-c)^2-10$. Then it is clear that, using the principal branch of the square root, $p^{-1}(0) = \{ c \pm (10/a)^{1/2} \}$ and we will denote these points by $z_1,z_2$. The pre-image $p^{-1}( S(0,10) )$ is a topological figure of eight, with $c$ at the crossing point, and enclosing $z_1$ and $z_2$ in different bounded components of the complement.  To see this, observe that $p(c) = -10$ and hence $c$ lies on $\gamma := p^{-1}( S(0,10) )$. The pre-image of a ray emanating from $-10$ is a straight line passing through $c$, noting that $\gamma$ is invariant under a rotation through $\pi$ about $c$. If the ray has angle $\theta$ strictly  between $-\pi/2$ and $\pi/2$ with respect to the semi-infinite line segment $[-10,\infty)$, then the pre-image intersects $\gamma$ twice. Otherwise it does not. As $\theta$ traverses the interval $(-\pi/2, \pi/2)$, $\gamma$ sweeps out a figure of eight with the claimed properties.

Moreover, the figure eight is oriented so that the straight line passing through $z_1,z_2$ and $c$ has slope $\tan( -\arg(a) / 2)$. This is since $z_1,z_2$ and $c$ clearly lie on the same straight line corresponding to the ray in the image with angle $\theta =0$. One can then check that the straight line $\sigma(t) = c+te^{-i\arg(a)/2}$ maps onto this ray under $p$.

\begin{lemma}
\label{lem:1}
Let $c\in \B^2$, $a\in \C$ and let $p(z) = a(z-c)^2-10$. For $\epsilon>0$, if $D_1$ denotes the disk $B(z_1,(\sqrt{1+\epsilon} - 1)|10/a|^{1/2} )$, then $p(D_1) \subset B(0,10\epsilon)$. The analogous statement holds for $D_2$.
\end{lemma}

\begin{proof}
Let $z = z_1 + \delta e^{it}$, where $\delta = (\sqrt{1+\epsilon} - 1)|10/a|^{1/2}$. Then
\begin{align*}
|p(z)| &= |p(z_1 + \delta e^{it})| \\
&= | a(  (10/a)^{1/2} + \delta e^{it} )^2 -10 | \\
&= \left | 2\delta e^{it} (10a)^{1/2} + a\delta^2 e^{2it} \right | \\
&\leq 2\delta |10a|^{1/2} + |a|\delta^2 \\
&= 20(\sqrt{1+\epsilon}-1) + 10(\sqrt{1+\epsilon} -1 )^2 \\
&=10\epsilon.
\end{align*}
\end{proof}

\subsection{The open set condition}

The open set condition (OSC) for an IFS is a particularly nice situation to be in. For example, if the OSC holds, one can make easy computations for the Hausdorff dimension of the attractor set. The OSC states that there is an open set $U$ so that 
\[ \varphi_i(U) \cap \varphi_j(U) = \emptyset\]
for $i\neq j$. In our situation, we do not make an assumption that the OSC holds. However, we will be able to pass to a similar property for an iterated version of the IFS.

Suppose that we have an IFS $X$ generated by conformal contractive mappings $\varphi_1,\ldots, \varphi_k$ on $\B^2$ with contraction factors $\lambda_1,\ldots, \lambda_k$ respectively. Since by assumption $\varphi_i(\overline{\B^2}) \subset \B^2$ for $i=1,\ldots, k$, the contraction mapping theorem states that $\fix(\varphi_i)$ is a unique point in $\B^2$. Each fixed point of $\varphi_i$ lies in the attractor set $S$.

For any $N\in N$, denote by $X^N$ the IFS generated by words of length $N$ in $\varphi_1,\ldots, \varphi_k$. 
For definiteness, denote such words by $\psi_1,\ldots, \psi_m$. Of course, $m$ could just be $k^N$, but it could be smaller if, for example, $\varphi_i$ and $\varphi_j$ commute.
The attractor set of $X^N$ is the same as the attractor set of $X$ since any infinite word consisting of elements of the $\varphi_i$ can also be composed of the elements of the $\psi_i$. This will allow us to pass from $X$ to $X^N$ where convenient, since we will not alter the attractor set.

\begin{lemma}
\label{lem:2}
Suppose that $X$ is generated by $\varphi_1,\ldots, \varphi_k$ and that $\#\{\fix(\varphi) : \varphi \in X \} \geq 2$. Given $\delta >0$, we can find $N\in \N$ so that if $X^N$ is generated by $\psi_1,\ldots, \psi_m$ then there exist $f_1,f_2 \in \{ \psi_1 , \ldots , \psi_m \}$ with the property that for any $j\in \{1,\ldots, m\}$, there exists $i\in \{1,2\}$ so that the distance between $f_i(\overline{\B^2})$ and $\psi_j(\overline{\B^2})$ is at least
\[  \left ( \frac{1}{ 2}- 2\delta \right )\max_{p\neq q} | \fix(\varphi_p) - \fix(\varphi_q) |.\]
\end{lemma}

In particular, we can ensure that $f_i(\overline{\B^2}) \cap \psi_j(\overline{\B^2})  = \emptyset$ with a definite distance between them.

\begin{proof}
Denote by $w_i$ the unique element of $\fix(\varphi_i)$.
Let $s = \max_i \{ \lambda_i\}<1$ and set $r =\max_{p\neq q} | w_p-w_q |$. By relabelling, we may assume this maximum is achieved for $\varphi_1,\varphi_2$. Given $\delta >0$, choose $N$ large enough so that $s^N < \delta r$. 
 Denote by $f_1,f_2$ the maps $\varphi_1^N,\varphi_2^N$ respectively. These two maps are different since they have different fixed points. Moreover, these maps are contained in the collection $\{\psi_1,\ldots, \psi_m\}$. By construction, $f_i(\overline{\B^2}) \subset B(w_i, \delta r)$ for $i=1,2$.

Now given $j\in \{1,\ldots, m\}$, the map $\psi_j$ has fixed point $w_j$ and we have $\psi_j(\overline{\B^2}) \subset B(w_j, \delta r)$. Since either $|w_1-w_j| \geq r/2$ or $|w_2-w_j| \geq r/2$, by choosing whichever is the larger, we conclude that
the balls $\overline{B(w_i,\delta r)}$ and $\overline{B(w_j, \delta r)}$ have distance at least $(1/2  - 2\delta )r$.
The lemma then follows.
\end{proof}

\subsection{Proof of Theorem \ref{thm:1} in dimension two}

Replace the IFS $X$ generated by $\varphi_1,\ldots, \varphi_k$ with $X^N$ generated by $\psi_1,\ldots, \psi_m$, recalling Lemma \ref{lem:2}. These both have attractor set $S$. Recall also the maps $f_1,f_2 \in \{\psi_1 , \ldots, \psi_m\}$. Since the fixed points $w_j$ of $\psi_j$ are contained in $\B^2$, it follows that there exists $T<1$ so that $\max_j |w_j| \leq T$.

For each $j\in \{1,\ldots, m\}$, consider $\psi_j$ with fixed point $w_j$ and the corresponding $f_i$, for $i\in \{1,2\}$, with fixed point $x_i$ so that the conclusions of Lemma \ref{lem:2} hold. 
We will construct a certain quadratic polynomial of the form $p_j(z) = a_j(z-c_j)^2-10$ so that $f_i(\overline{\B^2})$ and $\psi_j(\overline{\B^2})$ are small enough that they are each contained in disks $U,V$ of a definite size, which are in turn contained in each lobe of $p_j^{-1}(S(0,10))$ respectively. From here, we construct a quasiregular map which agrees with $f_i^{-1}$ and $\psi_j^{-1}$ on $f_i(\overline{B^2})$ and $\psi_j(\overline{\B^2})$ respectively, maps $\partial U$ and $\partial V$ onto $S(0,2)$ and maps the figure of eight $p_j^{-1}(S(0,10))$ onto $S(0,10)$.

To determine $p_j$, set $c_j=(w_j+x_i)/2$, choose $a_j$ so that $c_j\pm \sqrt{10/a_j}$ gives the two points $w_j$ and $x_i$, and if
\[ d:= d( \overline{B(x_i,\delta r)} ,\overline{ B(w_j , \delta r)} ),\]
we have $2|10/a_j|^{1/2}=d$. 

Let $U = B(x_i ,(\sqrt{6/5} - 1)|10/a_j|^{1/2} )$ and $V =  B(w_j ,(\sqrt{6/5} - 1)|10/a_j|^{1/2} )$. Now, given $\delta >0$ and $\epsilon = 1/5$ we first need to ensure, from Lemma \ref{lem:1}, that 
\[ \psi_j(\overline{\B^2}) \subset B(w_j,\delta r) \subset V \subset \B^2\]
for any $j\in \{1,\ldots, m\}$. Since $|w_j|\leq T<1$, this will be achieved as long as 
\[ T + (\sqrt{6/5}-1)|10/a_j|^{1/2} <1,\]
that is, if
\[ |a_j| > 10 \left ( \frac{ \sqrt{6/5}-1 }{1-T} \right )^2.\]
By the construction of $p_j$, we see that $p_j(\partial U)$ and $p_j(\partial V)$
are both $\partial B(0,2)$. Consequently, we need $\delta$ to be chosen small enough so that $B(w_j, \delta r) \subset V$.
By symmetry, we will also obtain that $B(x_i,\delta r) \subset U$.
We require $(\sqrt{6/5} - 1)d > 2\delta r$. Since $d \geq (1/2 - 2\delta )r$, recalling Lemma \ref{lem:2}, some elementary calculations show that we need
\[ \delta < \frac{1-\sqrt{5/6}}{4} \approx 0.0218.\]
In particular, if we choose $\delta <1/100$, we can ensure $B(w_j, \delta r) \subset V$.
Recalling that given $j \in \{1,\ldots, m\}$ we choose an appropriate $i\in \{1,2\}$, for $j\in \{1,\ldots, m\}$, we then define the degree two quasiregular mapping $g_j$ as follows:
\begin{equation}
g_j(z) = \left \{ \begin{array}{ll}
f_i^{-1}(z), &   z\in f_i(\overline{\B^2}),\\
\psi_j^{-1}(z), & z \in \psi_j(\overline{\B^2}) \\
p_j(z), & z\in \C \setminus (U\cup V)\\
q_j(z), & z\in (U  \setminus  f_i(\overline{\B^2}) ) \cup  (V \setminus  \psi_j(\overline{\B^2}) ) 
\end{array}
\right .
\end{equation}
where $q_j$ is a quasiconformal interpolation in two ring domains via Theorem \ref{thm:ann}. We note that $f_i$ and $\psi_j$ are injective on $\overline{\B^2}$ and so the inverses exist.

If we consider the semigroup generated by $g_1,\ldots, g_m$, then each $g_j$ is analytic outside either of the rings for the domain of definition of $q_j$. But $g_j$ maps these rings onto the annulus $\{ z:1<|z|<2\}$, whence when we apply any of the $g_i$, we are analytic. Therefore the orbit of any point passes at most once through an application of a map with non-trivial distortion. Consequently, the semigroup generated by $g_1,\ldots, g_m$ is a quasiregular semigroup $G$.

We finally need to show that $J(G)$ agrees with $S$. First, if $z\notin S$, then there exists a neighbourhood $U$ of $z$ and $k\in \N$ so that $U \cap X^k(\overline{\B^2}) = \emptyset$. Consequently, if $g\in G$ is any composition of at least $k+1$ elements from $\{g_1,\ldots, g_m \}$, $g(U) \cap\overline{\B^2} = \emptyset$. In particular, $G|_U$ is a normal family by Montel's Theorem. We conclude that $z\in F(G)$. 

On the other hand, if $z\in S$, then $z\in \bigcup_{j=1}^m \psi_j(\overline{\B^2})$. In particular, we can find an itinerary $(t_1,t_2,\ldots )$ for $z$ in $\{1,\ldots, m\}^{\N}$ as follows. Define $t_1\in \{1,\ldots , m\}$ such that $z\in \psi_{t_1}(\overline{\B^2})$. Note that $t_1$ may not be unique. Then define $t_2\in \{1,\ldots, m\}$ such that $\psi_{t_1}^{-1}(z) \in \psi_{t_2}(\overline{\B^2})$. We can continue inductively in this manner since $z\in S$ and $S$ is the attractor set. Then defining the sequence in $G$ given by $g_{t_1}, g_{t_2}\circ g_{t_1},\ldots$, we see that the orbit of $z$ under this sequence is bounded, since it must remain in $S$.
However, since $z\in \psi_{t_1}(\overline{\B^2})$, by construction, we can find $g_j$ from our generators of $G$ for which $\psi_{t_1}(\overline{\B^2})$ is in the escaping set $I(g_j)$. Here, we use the fact that $\psi_j(\overline{\B^2}) \cup f_i(\overline{\B^2}) \subset I(g_j)$. We conclude that there are sequences in $G$ for which the orbit of $z$ is both bounded and unbounded, respectively. Hence $z\in J(G)$.

This completes the proof of Theorem \ref{thm:1} in dimension two.

\subsection{Proof of Theorem \ref{thm:2}}

Let $E_i = \varphi_i(\B^2)$ for $i=1,\ldots, m$. By the hypotheses, $E_i$ is a quasidisk. Let $\epsilon >0$ be small.

{\bf Step 1:} Consider ring domains $R_i \subset \B^2 $ whose boundary components are smooth Jordan curves, so that $E_i$ is contained in the bounded component of $\B^2 \setminus \overline{R_i}$, the $R_i$ are pairwise disjoint and $E_j$ is not contained in the topological hull $\widetilde{R_i}$ of $R_i$ for $i\neq j$. We recall that the topological hull of a domain is the union of the domain with the bounded components of the complement. Denote by $D_i$ the bounded component of $\B^2 \setminus \overline{R_i}$.

Find $w_i \in E_i$ so that $B(w_i , \epsilon) \subset E_i$. There is a conformal map $\psi_i : D_i \to B(w_i,\epsilon)$.
Then by Theorem \ref{thm:ann}, we can find a quasiconformal map $g_i: R_i \to \widetilde{R_i} \setminus B(w_i,\epsilon)$ which is the identity on the outer boundary component of $R_i$ and equal to $\psi_i$ on the inner boundary component of $R_i$ (which equals $\partial D_i$). We then define $h_1:\B^2 \setminus \bigcup_{i=1}^m\overline{D_i} \to \B^2 \setminus \bigcup_{i=1}^m \overline{B(w_i,\epsilon)}$ as the quasiconformal map which is the identity except on $R_i$, where it equals $g_i$.

{\bf Step 2:} Let $\omega_1=1 ,\ldots, \omega_{m}$ denote the $m$'th roots of unity and set $v_i = (1-2\epsilon)\omega_i$ for $i=1,\ldots, m$. Let $F_1$ be a smooth Jordan domain in $\B^2$ compactly containing $B(w_1,\epsilon)$ from Step 1 and also $B(v_1,\epsilon)$, but not containing any of the disks $B(w_j,\epsilon)$ for $j=2,\ldots, m$. If this latter condition cannot be met, relabel either the $v_i$ or the $w_i$, or reduce $\epsilon$.
 
Let $T_1:B(w_1,\epsilon) \to B(v_1,\epsilon)$ be a translation and then set $p_1:F_1\setminus B(w_1,\epsilon) \to F_1 \setminus B(v_1,\epsilon)$ to be the quasiconformal map from Theorem \ref{thm:ann} which is the identity on $\partial F_1$ and $T_1$ on $\partial B(w_1,\epsilon)$. We obtain a quasiconformal map $q_1:\B^2 \to \B^2$ which moves $B(w_1,\epsilon)$ to $B(v_1,\epsilon)$ and leaves the other $B(w_j,\epsilon)$ alone. We repeat this procedure for $i=2,\ldots, m$ to obtain quasiconformal maps $q_1,\ldots, q_m$ so that when we form the composition $h_2 := q_m \circ \ldots \circ q_1$ we obtain a quasiconformal map which moves $B(w_i,\epsilon)$ to $B(v_i,\epsilon)$ for $i=1,\ldots, m$.

{\bf Step 3:} Letting $h_3(z) = z^m$, we see that $h_3$ maps $\B^2 \setminus \bigcup_{i=1}^m \overline{B(v_i,\epsilon)}$ onto a ring domain $R\subset \B^2$, as long as $\epsilon$ is chosen small enough.

{\bf Step 4:} Let $h_4$ be a quasiconformal map from $R$ onto the annulus $\{z : 3/2< |z| <2 \}$. By forming the composition $f=h_4\circ h_3 \circ h_2 \circ h_1$ in $\B^2 \setminus \bigcup_{i=1}^m \overline{D_i}$, we obtain a quasiregular map with image $\{ z : 3/2<|z|<2 \}$.

Since $f$ is a quasiconformal map in a neighbourhood of $\partial D_i$ in $\B^2 \setminus \bigcup_{i=1}^m \overline{D_i}$, we can apply Theorem \ref{thm:ann} to first extend $f$ to a quasiconformal map defined in $D_i \setminus \overline{E_i}$ by setting $f$ equal to $\varphi_i^{-1}$ on $\partial E_i$. We can then extend $f$ to be defined on all of $\B^2$ by setting it equal to $\varphi_i^{-1}$ on $E_i$.

{\bf Step 5:} for $|z|>2$, we set $f(z) = z^m$. The final task is then to interpolate so that $f$ is defined in $\{z:1<|z|<2\}$ with image $\{z:2<|z|<2^m\}$.

Let $\gamma$ denote the line segment $[2,2^m]$. Each endpoint of $\gamma$ has $m$ pre-images on $\{z:|z|=1\}$ and $\{z:|z| = 2\}$ respectively. Let $\sigma_i$, for $i=1,\ldots, m$, denote a hyperbolic geodesic in the hyperbolic domain $\{ z:1<|z|<2\}$ joining a pair of such pre-images, with the condition that the arcs $\sigma_i$ are pairwise disjoint. We then have a collection of simply connected domains $U_i$ in $\{z:1<|z|<2\}$, each with boundary made up of $\sigma_i, \sigma_{i+1}$ and two arcs of the boundary circles. We also have the domain $V = \{z:2<|z|<2^m\} \setminus \gamma$.

For each $i$, we construct a quasiconformal map $f_i:U_1 \to V$ as follows. We have a boundary map $f_i : \partial U_1 \to V$ defined via a bi-Lipschitz map from $\sigma_1$ and $\sigma_2$ onto $\gamma$ respectively, via the map from Step 1 on $\{z:|z|=1\}$ and via $z\mapsto z^m$ on $\{z: |z|=2\}$. Applying Riemann maps to $U_1$ and $V$, we obtain a quasisymmetric map $\partial \D \to \partial \D$ that we can extend inside $\D$ via any quasiconformal extension method, for example the Douady-Earle barycentric extension \cite{DE}. Applying the inverse of the Riemann maps, we obtain a quasiconformal map $f_i :U_1 \to V$. We repeat this argument inductively, except we are already given the map on $\sigma_i$ and have to choose a bi-Lipschitz map on $\sigma_{i+1}$. For the final step, we are given all of the boundary maps.

We then set $f$ to be $f_i$ on $\overline{U_i}$ for $i=1,\ldots,m$. The construction means the various $f_i$ agree on the overlaps of their closures of domains of definition, and so $f$ is a well-defined quasiregular map.

{\bf Conclusion:} By construction, if $z$ is in the attractor set $S$ of $X$, then $f$ acts on the orbit of $z$ always by conformal mappings. However, if $z$ does not lie on $S$, then $f$ acts on the orbit by at most finitely many of the conformal maps $\varphi_i^{-1}$, then the quasiregular map $h_4 \circ h_3 \circ h_2 \circ h_1$ from Step 4, then the quasiregular interpolation from Step 5 and finally from then on by $z\mapsto z^m$. Consequently $f$ is uniformly quasiregular. Since the escaping set is $I(f) = \R^2 \setminus S$, $S$ is a Cantor set and $\partial I(f) = J(f)$ for non-injective uqr maps by \cite[Lemma 5.2]{FN0}, we conclude that $J(f) = S$. This completes the proof.

\section{Dimension three}

\subsection{A degree two uqr map of power-type in $\R^3$}

To replace the quadratic polynomial from the dimension two case of the proof of Theorem \ref{thm:1}, we will use a degree two power-type map. This arises as a solution to the following Schr\"oder equation. Let $Z$ be a Zorich-type map which is strongly automorphic with respect to the group $G$ generated by $x\mapsto x+e_1$, $x\mapsto x+e_2$ and $x\mapsto \rho(x)$, where $e_1,e_2$ are the unit vectors in the $x_1,x_2$ direction and $\rho$ is a rotation through angle $\pi$ about the $x_3$-axis.
Let $A$ be a composition of a rotation through angle $\pi/4$ about the $x_3$-axis and a scaling by factor $\sqrt{2}$. Then since $AGA^{-1} \subset G$ there is a unique uqr solution $P$ to the Schr\"oder equation
\[ P\circ Z = Z\circ A,\]
see for example \cite{FM}. One can check that this uqr map is of power-type and has degree two. We remark that typically in the literature only power maps of degree $d^{n-1}$ have been constructed in $\R^n$, for some $d\in \N$. These arise when $A$ is a dilation, but in dimension three the inclusion of a rotation component to $A$ allows the degree to be reduced.

\subsection{Proof of Theorem \ref{thm:1} in dimension three}

Let $X$ be an IFS generated by $\varphi_1,\ldots, \varphi_m$.
The strategy of the proof is similar to that in dimension two, but some of the details are somewhat different. One aspect that is simpler in dimension three is that the only conformal mappings are M\"obius, so $\varphi_j(\B^3)$ is a ball for all $\varphi_j \in X$.

The same argument as in the dimension two case means we can replace $X$ by $X^N$ generated by $\psi_1,\ldots, \psi_m$ with the following property. There exist a constant $d>0$ and two maps $f_1,f_2 \in \{ \psi_1,\ldots, \psi_m \}$ with fixed points  $x_1,x_2$ respectively, so that for any $\psi_j$ with fixed point $w_j$, for some $i\in \{1,2\}$ we have
\[ d_h(w_j,x_i) \geq d,\]
where $d_h$ denotes the hyperbolic distance in $\B^3$. Moreover, we can assume that $N$ has been chosen large enough that 
\[ d_h( f_i(\overline{\B^3})  , \psi_j ( \overline{\B^3} ) ) \geq \frac{d}{10}. \]
Denoting by $B_h(x,r)$ the hyperbolic ball centred at $x\in \B^3$ of radius $r>0$, we can find $C>0$ so that $f_i(\overline{\B^3}) \subset B_1:=B_h(x_i,C)$, $\psi_j(\overline{\B^3}) \subset B_2:=B_h(w_j,C)$ and $d_h(B_1,B_2) \geq d/2$.

Then apply a M\"obius map $M:\B^3\to \B^3$ which sends the hyperbolic midpoint of $x_i$ and $w_j$ to $0$, and so that $M(x_i)$ and $M(w_j)$ are on the axis in the direction of the $e_1$ basis vector.

Since $M(\widetilde{B_k})$, for $k=1,2$, have the same dimensions in the Euclidean metric, that is, they are given by $B(\pm te_1,s)$ for some $s,t>0$, we can apply a winding map of the form $W:(r,\theta,z) \mapsto (r,2\theta,z)$ in cylindrical polar coordinates which maps $\B^3$ onto itself in a two-to-one manner and so that $W$ applied to $M(\widetilde{B_1})$ and $M(\widetilde{B_2})$ have the same image. Hence the image of $\B^3 \setminus (M(\widetilde{B_1}) \cup M(\widetilde{B_2}) )$ under $W$ is a ring domain. By Lemma \ref{lem:ringdomain}, we can then apply a quasiconformal map $h$ which maps this ring domain onto the annular ring domain $\{ x: 3/2<|x| < 2\}$.

We then define the map $g_i$ as follows:
\begin{itemize}
\item In $f_i(\overline{\B^3}),\psi_j(\overline{\B^3)}$, we set $g_i$ to be $f_i^{-1}$ and $\psi_j^{-1}$ respectively, with image $\overline{\B^3}$.
\item In $\overline{\B^3} \setminus ( B_1 \cup B_2)$, we set $g_i$ to be $h\circ W\circ M$. The image is $\{x : 3/2 \leq |x| \leq 2 \}$.
\item In $B_1 \setminus f_i(\overline{\B^3})$ and $B_2 \setminus \psi_j(\overline{\B^3})$, we use Theorem \ref{thm:ann} to find quasiconformal interpolations, with image $\{x : 1<|x|<3/2 \}$.
\item For $|x|>2$, we set $g_i$ to be the degree two uqr power map discussed in the previous section.
\item Finally, for $1<|x|<2$, we interpolate by using the Berstein-Edmonds extension theorem \cite[Theorem 6.2]{BE} exactly as in the proof of \cite[Theorem 1.1]{FW}.
\end{itemize}

Each $g_i$ is uqr by construction, since an orbit consists either of only conformal maps, or by finitely many conformal maps then at most three quasiregular maps and finally repeated application of a uqr map. As in the dimension two case, $G = \langle g_1,\ldots, g_m \rangle$ is a quasiregular semigroup. To see that $J(G)=S$ we apply the same reasoning as in the dimension two case. Briefly, if $x\notin S$, then there is a neighbourhood $U$ of $x$ on which any sequence from $G$ converges uniformly to infinity and hence $x\in F(G)$. On the other hand, if $x\in S$, then we can find two sequences of elements of $G$, one for which the orbit of $x$ is unbounded and one for which the orbit of $x$ remains in $S$, and hence is bounded. This completes the proof.

\end{document}